\documentclass{amsart}%
\usepackage{amsfonts}%
\usepackage{amsmath}%
\usepackage{amssymb}%
\usepackage{graphicx}
\newtheorem{theorem}{Theorem}
\theoremstyle{plain}

\newtheorem{definition}{Definition}

\newtheorem{lemma}{Lemma}

\newtheorem{proposition}{Proposition}
\newtheorem{remark}{Remark}

\numberwithin{equation}{section}
\begin{document}
\title[Short Title]{The reducibility Of An Airy Operator}
\author{Lotfi Saidane}
\address{D\'{e}partement de Math\'{e}matiques, Facult\'{e} des Sciences de Tunis,
Universit\'{e} de Tunis-El Manar, Campus universitaire, 2092, El-Manar, Tunis ,\ TUNISIA}
\email{lotfi.saidane@fst.rnu.tn.}
\date{Febrary 2010}
\subjclass{[2000]Primary 12 H 05, 33C 10, 11C20, 15B36}
\keywords{Airy operator, Setoyanagi operator, Determining factor, Differential equation, Determinant.}

\begin{abstract}
We show that the determinant $\nabla(d,\alpha),$ which seems to be not
considered in the past, is not zero. As an application of this result we
prove that the Setoyanagi operator $S_{p,q}=\partial^{2}-\left(
ax^{p}+bx^{q}\right)  $ is irreducible over $\mathbb{C}\left(  x\right)
\left[  \partial\right]  $.

\end{abstract}
\maketitle

\section{Introduction}

It is well known (see [D, L-R]) that is the operator $\partial^{2}-q,$
$q\in\mathbb{C}\left[  x\right]  $ is reducible in $\mathbb{C}(x)[\partial]$
if and only if the Ricatti equation $u^{\prime}+u^{2}=q$ has a solution in
$\mathbb{C}(x)$. We propose a more manageable criterion using the determining
factors properties of the operator. We can easily prove that an Airy operator
$L=\sum_{i=0}^{n}a_{i}\partial^{i}+Q_{m}\left(  x\right)  $, of bidegree
$\left(  n,m\right)  ,$ $n\leq m,$ and $\left\{  \int R_{i}\left(
x^{1/n}\right)  dx,\ i=1,..,n\right\}  $ as set of determining factor is
reducible in $\mathcal{D}_{K}$ with a right factor of order $1$ if and only
if, $n$ divides $m$ and there exist $i$ in $\left\{  1,...n\right\}  $ such
that the differential equation $L(\partial+R_{i})(u)=0$ has a polynomial
solution. As an application, we prove that the Setoyanagi operator
$S_{p,q}=\partial^{2}-\left(  ax^{p}+bx^{q}\right)  $ (see [D, L-R], Example
4) is irreducible over $\mathbb{C}\left(  x\right)  \left[  \partial\right]
.$ For this purpose, we show that the determinant $\nabla(d,\alpha)=\left\vert
c_{i,j}\right\vert _{1\leq i,j\leq d+1},$ defined in section 4, which seems to
be not considered in the past, is not zero.

\section{Determining Factors}

Let $k$ be an algebraically closed field of characteristic zero, $K$ is the
quotient field of the ring of polynomials $R=k[x],$ $\partial=\frac{d}{dx}$ is
the derivation of $K$, $\mathcal{D}=R[\partial]=k[x,\partial]$ is the Weyl
algebra over $k$ and $\mathcal{D}_{K}=K[\partial]$ is the set of differential
operators with coefficients in $K$, so $\mathcal{D}_{K}$ is an associative
noncommutative $k$-algebra. We denote, also, by $\partial$ the extension of
$\partial$ to the Picard-Vessio extensions of $K$, $V$ an $R$-module of rank
$n$, and $\nabla_{\partial}$ a contraction by $\partial$ of a connection on
$V$, i.e. a $k$-linear map on $V$ satisfying: :%
\[
\nabla_{\partial}\left(  av\right)  =\left(  \partial a\right)  v+a\nabla
_{\partial}\left(  v\right)  ,\ a\in R,\ v\in V.
\]
We define a structure of left $\mathcal{D}$-module on $V$ by setting :%
\[
(\sum_{i=0}^{n}a_{i}\partial^{i})v=\sum_{i=0}^{n}a_{i}\nabla_{\partial}%
^{i}\left(  v\right)  ,\ \ a_{i}\in R,\ \ v\in V
\]
Inversely, if $V$ is a left $\mathcal{D}$-module of finite rank as an
$R$-module, than we can define a connection on it by putting $\nabla\left(
x\right)  =\partial\left(  x\right)  $.

An operator $L\in\mathcal{D}$ is said to be monic if its leading coefficient,
with respect to $\partial$ is one. If $L$ is monic operator than $\mathcal{D}%
$/$\mathcal{D}L$ is a free $R$-module of finite rank. We say that a $D$-module
$V$ is cyclic if there exist a monic operator $L$ in $\mathcal{D}$ such that
$V$ is isomorphic, as a $D$-module to, $\mathcal{D}$/$\mathcal{D}L$. By scalar
extension we can define $V_{K}=K\otimes_{R}V,$ thus $V_{K}$ is a $K-$vector
space of dimension the rank of $V$ as a free $R$-module. As above, we define a
connection on $V_{K}$, denoted by $\nabla_{K} $, or simply $\nabla,$ as
follows :%
\[
\nabla\left(  a\otimes v\right)  =\left(  \partial a\right)  \otimes
v+a\otimes\nabla\left(  v\right)  ,\ \ \ \ a\in K,\ v\in V.
\]
thereby $\nabla$ define a structure of $\mathcal{D}_{K}$-module on $V_{K}$.

\begin{definition}
Two operators of $\mathcal{D}$ (resp. $\mathcal{D}_{K}$) are said equivalent
if their corresponding $D$-modules are equivalent.
\end{definition}

The following properties (which we can find a proof in [Sg], \S \ 2) will be
useful for the rest of this paper.

\begin{proposition}
Two monic operators $L_{1}$ and $L_{2}$ of $\mathcal{D}_{K}$ are equivalent if
and only if there exists $L_{3}$ in $\mathcal{D}_{K}$ having no common factor
on the right with $L_{1}$ and $L_{4}$ in $\mathcal{D}_{K}$ such that
$L_{2}\circ L_{3}=L_{4}\circ L_{1}$.
\end{proposition}

Let $\tilde{K}$ be a Picard-Vessiot extension of $K$, containing the
Picard-Vessiot extension of $L_{1}$ and $L_{2}$. We denote by $V_{1},$ $V_{2}
$ the respective solutions spaces of of $L_{1}(y)=0$ and $L_{2}(y)=0$. Let $G
$ be the differential Galois group of $\tilde{K}$. Then, the operators $L_{1}
$ and $L_{2}$ are equivalent if and only if their corresponding $G$-modules
$V_{1}$ and $V_{2}$ are isomorphic (the isomorphism is given by the natural
action on $V_{1}\subset\tilde{K}$. The operator $L_{3}$ can be chosen such
that its order is strictly lower than that of $L_{1}$, see [Sg], Lemma 2.5.

\begin{definition}
An element of $\mathcal{D}_{K}$ is called reducible (resp. completely
reducible) if it decomposes into a product of at least two factors of order
$\geq1$ (resp. of order $1$).
\end{definition}

An operator of $\mathcal{D}$ may be reducible in $\mathcal{D}_{K}$ without
being on $\mathcal{D}$ (see [Be] \S \ 3. )).

Let $\hat{K}=k((1/x))$ be the field of meromorphic formal series, near the
infinity, with coefficients in $k$ ($k$ is an algebraically closed field of
characteristic zero), equipped with its usual derivation $\partial=\frac
{d}{dx}$ and its valuation $1/x$-adic $v$. Let $\bar{K}$ be the algebraic
closure of $\hat{K}$, then the valuation and the derivation of $\hat{K}$
extends uniquely to $\bar{K}$. For example if $a\in k[x]$ is a polynomial of
degree $deg(a)$, $v(a)=-deg(a)$.
%TCIMACRO{\U{f020} }%
%BeginExpansion
\protect\rule{0.1in}{0.1in}
%EndExpansion
More generally, we will put $deg(a)=-v(a)$ for all $a\in\bar{K}$. Let%
\begin{equation}
L=\partial^{n}+\sum_{i=0}^{n-1}a_{i}\partial^{i},\ \ \ \ a_{i}\in\hat
{K}\label{eq}%
\end{equation}
be a differential operator with coefficients in $\hat{K}$. The theorem and
Hukuhara Turrittin (see [I]) shows the existence of a basis ($u_{1},...,u_{n}
$) of solutions of the form:%
\begin{equation}
u_{i}(x)=(expK_{i}(x))(1/x)^{\lambda_{i}}v_{i}%
(1/x),\ \ \ \ \ \ \ i=1,...,n,\label{eq2}%
\end{equation}
where $K_{i}(x)$ is a polynomial in $x^{1/q}$ (for some integer $q$) without
constant term, $\lambda_{i}$ is an element of $k$, $(1/x)^{\lambda_{i}}$ is a
solution (in a Picard-Vessiot extension of $\hat{K}$) of the differential
equation $y^{\prime}=-(\lambda_{i}/x)y$ and%
\[
v_{i}(x)=\sum_{j=0}^{n_{i}}v_{i,j}\left(  1/x\right)  \left(  Logx\right)
^{j},
\]
where%
\[
v_{i,j}\left(  x\right)  =\sum_{k=0}^{\infty}v_{i,j,k}x^{-k/q}\in\bar{K}.
\]

\begin{definition}
The polynomials $K_{i},$ $i=1,...,n$, \ref{eq2}, are called the determining
factors of the operator $L$ \ref{eq}. We say that $L$ is of simple
characteristics, if for every pair $(i,j)$ of distinct indices in $\left\{
1,2,...,n\right\}  ,$ we have $\deg(K_{i}-K_{j})=\deg(K_{i}).$
\end{definition}

\begin{definition}
Let $n,$ $m\in\mathbb{N}$, $n\neq0,$ $P_{n},$ $Q_{m}$ two polynomials in
$k\left[  x\right]  $ of degree $n,$ $m$ and $\partial=\frac{d}{dx}.$ The
operator $L=P_{n}\left(  \partial\right)  +Q_{m}\left(  x\right)  $ is called
an Airy operator of bidegree $\left(  n,m\right)  .$
\end{definition}

Airy operator generalize the classical Airy equation $y"-xy=0$ (where $n=2$
and $m=1$). Their study was initiated by N. Katz [K], in order to calculate
the differential Galois group. Katz shows that this calculation is reduced
precisely to questions of reducibility and self duality for the operator $L$.

\begin{lemma}
The characteristics of an Airy operator $L=P_{n}\left(  \partial\right)
+Q_{m}\left(  x\right)  $ of bidegree $\left(  n,m\right)  $ is simple and the
determining factors $K_{i}$ are polynomials in $x^{1/n}$ of degree $\frac
{n+m}{n}$ without constant terms. Their derivatives $K_{i}^{\prime}=R_{i}$
are, in the case $m\geq n,$ solutions to the inequality%
\[
\deg\left(  P_{n}\left(  R\right)  +Q_{m}\left(  x\right)  \right)  \leq
\frac{nm-m-n}{n}.
\]
where $R\in\bar{K}$ is a polynomial in $x^{1/n}$. \cite{lem1}
\end{lemma}

\begin{proof}
See [S], Lemma 1.2.1 page 525.
\end{proof}

\begin{definition}
The Fourier transform of an operator $L=\sum_{i=0}^{N}a_{i}\left(  x\right)
\partial^{i},\ a_{i}\in R$ of $\mathcal{D}$ , relatively to $\partial$, is the
operator $\mathcal{F}\left(  L\right)  =\sum_{i=0}^{N}a_{i}\left(
\partial\right)  \left(  -x\right)  ^{i}.$
\end{definition}

\begin{remark}
The Fourier transform $\mathcal{F}$ is a $k$-linear bijection from
$\mathcal{D}$ to itself. We have
\[
\mathcal{F}^{2}\left(  L\right)  =\left[  -1\right]  ^{\ast}L==\sum_{i=0}%
^{N}a_{i}\left(  -x\right)  \left(  -\partial\right)  ^{i}.
\]

\end{remark}

\begin{proposition}
$L\in\mathcal{D}$ is reducible (in $\mathcal{D)}$ if and only if its Fourier
transform $\mathcal{F}\left(  L\right)  $ is reducible.
\end{proposition}

An operator $L\in\mathcal{D}$ is called biunitary if its leading coefficient
relatively to $\partial$ and $x$ are unity of $R.$ The following result
improve the previous.

\begin{proposition}
A biunitary operator $L\in\mathcal{D}$ is reducible (in $\mathcal{D}_{K}$ if
and only if its Fourier transform $\mathcal{F}\left(  L\right)  $ is reducible
in $\mathcal{D}_{K}.$
\end{proposition}

\begin{proof}
See [S], Proposition 1.3.1.
\end{proof}

\section{Reducibility}

Katz ([Ka], page 26), has proved that an Airy operator of bidegree $(n,m),$
where $n$ and $m$ are coprime, is irreducible in $\mathcal{D}_{K}.$ The
following result improve the preview.

\begin{theorem}
An Airy operator $L=\sum_{i=0}^{n}a_{i}\partial^{i}+Q_{m}\left(  x\right)  $,
of bidegree $\left(  n,m\right)  ,$ $n\leq m,$ and $\left\{  \int R_{i}\left(
x^{1/n}\right)  dx,\ i=1,..,n\right\}  $ as set of determinig factor is
reducible in $\mathcal{D}_{K}$ with a right factor of order $1$ if and only if
the following two two conditions are satisfied :\newline1) $n$ divides
$m$,\newline2) there exist $i$ in $\left\{  1,...n\right\}  $ such that the
differential equation $L(\partial+R_{i})(u)=0$ has a polynomial
solution.\label{th1}
\end{theorem}

\begin{proof}
See [S] Proposition 2.1.1.
\end{proof}

If $L=\sum_{i=0}^{n}a_{i}\partial^{i}+Q_{m}\left(  x\right)  ,$ then the
adjoint operator of $L$, denoted $L^{\nu},$ is defined by%
\[
L^{\nu}=\sum_{i=0}^{n}\left(  -\partial\right)  ^{i}a_{i}+Q_{m}\left(
x\right)  .
\]
With the same hypothesis as the theorem, we can easily prove that $L$ is
reducible in $\mathcal{D}_{K}$ with a left factor of order $1$ if and only if
$n$ divides $m$ and there exist $i$ in $\left\{  1,...n\right\}  $ such that
the differential equation $L^{\nu}(\partial-R_{i})(u)=0$ has a polynomial solution.

It is well known (see [D, L-R]) that the operator $\partial^{2}-q,$
$q\in\mathbb{C}\left[  x\right]  $ is reducible in $\mathbb{C}(x)[\partial]$
if and only if the Ricatti equation $u^{\prime}+u^{2}=q$ has a solution in
$\mathbb{C}(x)$. We propose a more manageable criterion using the determining
factor properties of the operator.

\section{Application}

Let $p,$ $q\in\mathbb{N},$ $q<p,$ $a,$ $b\in\mathbb{C},$ $ab\neq0.$ Let
$S_{p,q}=\partial^{2}-\left(  ax^{p}+bx^{q}\right)  $ the
%TCIMACRO{\TEXTsymbol{<}}%
%BeginExpansion
$<$%
%EndExpansion
Setoyanagi operator. If we assume that the operator $S_{p,q}$ is reducible in
$\mathbb{C}(x)[\partial]$, then the Theorem \ref{th1} implies that $p$ is an
even integer. We assume $p=2m$. The determining factors $\int R$ of $S_{p,q}$,
according to Lemma \ref{lem1}, are given by:%
\[
R=\varepsilon\sqrt{a}x^{m}\sum_{i=0}^{r}\binom{1/2}{i}\left(  b/a\right)
^{i}x^{i(q-2m)}%
\]
where $r=E\left(  \frac{m}{2m-q}\right)  $ is the integral part of $\frac
{m}{2m-q}$ and $\varepsilon=\pm1.$ Using Theorem \ref{th1}, we deduce that
there exists $d\in\mathbb{N}$ such that if $d_{m-1}$ denote the coefficient of
$x^{m-1}$ in $R^{2}-[ax^{2m}+bx^{q}]$ so we obtain%
\[
\frac{d_{m-1}}{\left(  2d+m\right)  \sqrt{a}}=\varepsilon,\ with\ \varepsilon
=\pm1.
\]
We therefore find the conditions cited by [D, L-R], namely, $\frac{m+1}{2m-q}
$ is a natural number $s\geq1$ (in fact $(r+1)(q-2m)+2m$ is equal to $m-1$)
and condition%
\[
d_{m+1}=2a\binom{1/2}{s}\left(  b/a\right)  ^{s}=\varepsilon\left(
2d+m\right)  \sqrt{a},
\]
we can write as follows: there exist $d\in\mathbb{N}$ such that%
\[
\varepsilon\sqrt{a}\binom{1/2}{s}\left(  b/a\right)  ^{s}-\frac{m}{2}=d.
\]
Setoyanagi (cited by [D, L-R], Example 4) gave a necessary and sufficient
conditions of reducibility in the case $\frac{m+1}{2m-q}=s\leq2$. For $s=2$,
we can show that this case reduces to the case $m=q=1$.

As an application, we will consider $m=2$ and $q=3$ (so $s=3$), we show the
following result:

\begin{proposition}
The Setoyanagi operator $S_{4,3}=\partial^{2}-\left(  ax^{4}+bx^{3}\right)  ,$
$a,$ $b\in\mathbb{C},$ $a$ or $b\neq0,$ is irreducible in $\mathbb{C}%
(x)[\partial].$
\end{proposition}

\begin{proof}
If $a$ or $b=0,$ it is easy to verify that $S_{4,3}$ is irreducible. For the
following, we assume that $ab\neq0$. The change of variable $x-\frac{b}{4a}$
preserves the reducibility properties of the operator $S_{4,3}$. Let $S$ be
the operator obtained after the change of variable. We have%
\[
S=\partial^{2}-Q,\
\]
where$\ $%
\[
Q=ax^{4}-\frac{3b^{2}}{8a}x^{2}+\frac{b^{3}}{8a^{2}}x-\frac{3b^{4}}{4^{4}%
a^{3}}.
\]
The determining factors of $S$ are $\int R$ with%
\[
R=\varepsilon\sqrt{a}\left[  x^{2}-\frac{3b^{2}}{16a^{2}}\right]
,\ \ \varepsilon=\pm1.
\]
The operator $S^{R}=S\left(  \partial+R\right)  $ is equal to $\partial
^{2}+2R\partial+\left[  R^{2}+R^{\prime}-Q\right]  .$ $If$ $S$ is reducible
than the differential equation $S^{R}\left(  u\right)  =0$ have a polynomial
solution. However, there exist $d\in\mathbb{N}$ such that%
\[
\frac{b^{3}}{16sa^{3/2}}=d+1.
\]
As above we can suppose $s=1.$ If necessary, we change the argument of
$\sqrt{a}.$ We put $\alpha=\frac{2a}{b}$, consequently
\[
\sqrt{a}=2\left(  d+1\right)  \alpha^{3}%
\]
and%
\[
S^{R}=\partial^{2}+\left[  4\left(  d+1\right)  \alpha^{3}x^{2}-3\left(
d+1\right)  \alpha\right]  \partial+\left[  -4d\left(  d+1\right)  \alpha
^{3}x+3\left(  d+1\right)  ^{2}\alpha^{2}\right]
\]
The differential equation $S^{R}\left(  u\right)  =0$ have a polynomial
solution of degree $d$ if and only if $S^{R}$, considered as a linear operator
on $\mathbb{C}_{d}\left[  x\right]  $ is not an injection. This is equivalent
to determinant $\nabla(d,\alpha)=\left\vert c_{i,j}\right\vert _{1\leq i,j\leq
d+1}$, defined by%
\[
c_{i,j}=\left\{
\begin{array}
[c]{c}%
3\left(  d+1\right)  ^{2}\alpha^{2}%
,\ \ \ \ \ \ \ \ \ \ \ \ \ \ \ \ \ \ \ \ \ \ \ \ \ \ \ \ \ \ \ \ \ \ \ \ \ \ \ \ \ \ \ \ \ \ if\ j=i\\
-4\left(  d-i+2\right)  \left(  d+1\right)  \alpha^{3}%
,\ \ \ \ \ \ \ \ \ \ \ \ \ \ \ \ \ \ \ \ \ \ \ \ if\ \ \ j=i-1\\
-3i\left(  d+1\right)  \alpha
,\ \ \ \ \ \ \ \ \ \ \ \ \ \ \ \ \ \ \ \ \ \ \ \ \ \ \ \ \ \ \ \ \ \ \ \ \ \ \ \ \ \ \ if\ \ j=i+1\\
i.\left(  i+1\right)
,\ \ \ \ \ \ \ \ \ \ \ \ \ \ \ \ \ \ \ \ \ \ \ \ \ \ \ \ \ \ \ \ \ \ \ \ \ \ \ \ \ \ \ \ \ \ \ \ \ if\ j=i+2\\
0\ \ \ \ \ \ \ \ \ \ \ \ \ \ \ \ \ \ \ \ \ \ \ \ \ \ \ \ \ \ \ \ \ \ \ \ \ \ \ \ \ \ \ \ \ \ \ \ \ \ \ \ \ \ \ \ \ \ \ \ \ \ \ \ \ \ \ \ \ \ else
\end{array}
\right.
\]

is zero.
\end{proof}

\begin{lemma}
$\nabla(d,\alpha)\neq0.$
\end{lemma}

\begin{proof}
$\nabla(d,\alpha)$ is a polynomial in two variables $d$ and $\alpha$,
homogeneous of degree $2(d+1)$ compared to $\alpha.$ Therefore%
\[
\nabla(d,\alpha)=\mu\left(  d\right)  \alpha^{2\left(  d+1\right)  }.
\]
It suffices to show that $\mu\left(  d\right)  \neq0$. For this purpose, we
note that if $d$ is an even integer then $\mu\left(  d\right)  $ is congruent
to $1$ modulo $2$ and $\mu\left(  d\right)  $ is not zero. Assume, for the
rest that $d$ is an odd integer. Let $\left(  a_{i,j}\right)  $ the matrix
obtained from $\left(  c_{i,j}\right)  $ after putting $\alpha=1$. Then%
\[
\left\vert a_{i,j}\right\vert =\mu\left(  d\right)  .
\]
We consider the order $p$ determinants extracted from $\left(  a_{i,j}\right)
$ as follows:%
\[
\nabla_{p}=\left\vert a_{i,j}\right\vert _{d-p+2\leq i,j\leq d+1},
\]
for example%
\begin{align*}
\nabla_{2}  & =%
\begin{vmatrix}
3\left(  d+1\right)  ^{2} & -3d\left(  d+1\right) \\
-4\left(  d+1\right)  & 3\left(  d+1\right)  ^{2}%
\end{vmatrix}
\\
& =3\left(  d+1\right)  ^{2}\left(  3d^{2}+2d+3\right)
\end{align*}
and%
\begin{align*}
\nabla_{0}  & =1,\ \ \ \\
\nabla_{d+1}  & =\mu\left(  d\right)  .
\end{align*}
Consequently, after developing $\ \nabla_{d+1}$ with respect to its first
colon, we obtain the recurrence relation%
\begin{equation}
\nabla_{p+1}=\left(  d+1\right)  ^{2}\left[  3\nabla_{p}-12\lambda_{p-1}%
\nabla_{p-1}+16\lambda_{p-1}\lambda_{p-2}\nabla_{p-2}\right]  ,\label{s}%
\end{equation}
for $p=2,$ .., $d,$ and%
\[
\lambda_{k}=\left(  k+1\right)  \left(  d-k\right)  .
\]
We propose to prove that, for $p\in\left\{  1,..,d-1\right\}  $%
\[
\nabla_{p+1}>4\lambda_{p}\nabla_{p}.
\]
For the rest we put $y=3\left(  d+1\right)  ^{2},$ then%
\begin{equation}
\left\{
\begin{array}
[c]{c}%
\nabla_{1}%
=y\ \ \ \ \ \ \ \ \ \ \ \ \ \ \ \ \ \ \ \ \ \ \ \ \ \ \ \ \ \ \ \ \ \ \ \ \ \ \ \ \ \ \ \ \ \ \ \ \ \ \ \ \ \ \\
\nabla_{2}=\left(  y-4\lambda_{0}\right)  \nabla_{1}>4\nabla_{1}%
\ \ \ \ \ \ \ \ \ \ \ \ \ \ \ \ \ \ \ \ \ \ \ \ \ \ \ \ \\
\nabla_{3}=y\left(  y^{2}-\left(  4\lambda_{0}+4\lambda_{1}\right)
y+\frac{16}{3}\lambda_{0}\lambda_{1}\right)  >4\lambda_{2}\nabla_{2}.
\end{array}
\right. \label{s1}%
\end{equation}
We assume that there exists an integer $q\in\left\{  1,...,d-1\right\}  $ such
that%
\[
\nabla_{q+1}\leq4\lambda_{q}\nabla_{q}.
\]
Let $p=\inf\left\{  q;\ \nabla_{q}\leq4\lambda_{q}\nabla_{q}\right\}  .$ The
system \ref{s1} leads that $p\geq4.$\newline By writing%
\[
\nabla_{p}>4\lambda_{p-1}\nabla_{p-1}%
\]
and%
\[
\nabla_{p+1}\leq4\lambda_{p}\nabla_{p},
\]
we can deduce the following relation%
\[
\left(  y-4\lambda_{p}\right)  \nabla_{p}\leq4\lambda_{p-1}y\left(
\nabla_{p-1}-\frac{4}{3}\lambda_{p-2}\nabla_{p-2}\right)  ,
\]
where%
\begin{equation}
\nabla_{p}\leq\frac{4\lambda_{p-1}y}{y-4\lambda_{p}}\left[  \nabla_{p-1}%
-\frac{4}{3}\lambda_{p-2}\nabla_{p-2}\right]  .\label{s2}%
\end{equation}
The relations%
\[
\nabla_{p}=y\left[  \nabla_{p-1}-4\lambda_{p-2}\nabla_{p-2}+\frac{16}%
{3}\lambda_{p-2}\lambda_{p-3}\nabla_{p-3}\right]  ,
\]
and \ref{s2} leads%
\begin{equation}
y\left(  1-\frac{4\lambda_{p-1}}{y-4\lambda_{p}}\right)  \nabla_{p-1}%
\leq4\lambda_{p-2}y\left[  1-\frac{\frac{4}{3}\lambda_{p-1}}{y-4\lambda_{p}%
}\right]  \nabla_{p-2}-\frac{16}{3}\lambda_{p-2}\lambda_{p-3}y\nabla
_{p-3}.\label{s3}%
\end{equation}
Replacing $\nabla_{p-1}$ by its expression in the inequality \ref{s3} we
obtain%
\begin{align}
& \left[  y^{2}\left(  1-\frac{4\lambda_{p-1}}{y-4\lambda_{p}}\right)
-4\lambda_{p-2}y\left[  1-\frac{\frac{4}{3}\lambda_{p-1}}{y-4\lambda_{p}%
}\right]  \nabla_{p-2}\right] \label{s4}\\
& \leq\left[  4\lambda_{p-3}y^{2}\left(  1-\frac{4\lambda_{p-1}}%
{y-4\lambda_{p}}\right)  -\frac{16}{3}y\lambda_{p-2}\lambda_{p-3}y\right]
\nonumber\\
& \nabla_{p-3}-\frac{16}{3}y^{2}\left(  1-\frac{4\lambda_{p-1}}{y-4\lambda
_{p}}\right)  \lambda_{p-3}\lambda_{p-4}y\nabla_{p-4}.\nonumber
\end{align}
From the definition of $p$, we deduce that $\nabla_{p-3}>4\lambda_{p-4}%
\nabla_{p-4}$, and the relationship \ref{s4} then leads%
\begin{align}
& y\left[  y\left(  1-\frac{4\lambda_{p-1}}{y-4\lambda_{p}}\right)
-4\lambda_{p-2}\left[  1-\frac{\frac{4}{3}\lambda_{p-1}}{y-4\lambda_{p}%
}\right]  \nabla_{p-2}\right] \label{s5}\\
& <4\lambda_{p-3}y\left[  \frac{2}{3}y\left(  1-\frac{4\lambda_{p-1}%
}{y-4\lambda_{p}}\right)  -\frac{4}{3}\lambda_{p-2}\right]  \nabla
_{p-3}.\nonumber
\end{align}
Similarly, the relationships%
\[
\nabla_{p-2}>4\lambda_{p-3}\nabla_{p-3}%
\]
and \ref{s5} leads%
\[
y\left(  y-4\lambda_{p}+4\lambda_{p-1}\right)  +4\lambda_{p-2}\left(
y-4\lambda_{p}\right)  -12\lambda_{p-2}\left(  y-4\lambda_{p}-\frac{4}%
{3}\lambda_{p-1}\right)  <0.
\]
Thereby, we have%
\begin{equation}
y^{2}-y\left(  4\lambda_{p}+4\lambda_{p-1}+8\lambda_{p-2}\right)
+32\lambda_{p}\lambda_{p-2}+16\lambda_{p-1}\lambda_{p-2}<0\label{s6}%
\end{equation}
We put, for $p\in\left\{  2,..,d-1\right\}  ,$%
\[
h\left(  p\right)  =y^{2}-y\left(  4\lambda_{p}+4\lambda_{p-1}+8\lambda
_{p-2}\right)  +32\lambda_{p}\lambda_{p-2}+16\lambda_{p-1}\lambda_{p-2}%
\]
we differentiate the function $h$ and we replace replace $\lambda_{k}$ by
$\left(  k+1\right)  \left(  d-k\right)  ,$ thus%
\begin{align*}
h^{\prime}\left(  p\right)   & =8\left(  d-p+3\right)  \left(  4\lambda
_{p}+2\lambda_{p-1}-y\right)  +4\left(  d-2p-1\right)  \left(  8\lambda
_{p-2}-y\right) \\
& +4\left(  d-2p+1\right)  \left(  4\lambda_{p-2}-y\right)  .
\end{align*}
For $k\in\left\{  0,1,..,d-1\right\}  $ the function $\lambda_{k}$ varies
between $d$ and $\frac{\left(  d+1\right)  ^{2}}{2}$. As $y=3(d+1)^{2},$ the
sign of $h^{\prime}$ is strictly positive if $p>\frac{d+3}{2}$; is strictly
negative if $p<\frac{d-1}{2}$ (as $d$ is an odd integer $d$, view the
beginning of the prof, and $p$ is an integer). We calculate $h\left(
\frac{d-1}{2}\right)  ,$ $h\left(  \frac{d+1}{2}\right)  $ and $h\left(
\frac{d+3}{2}\right)  $. We find%
\begin{align*}
h\left(  \frac{d-1}{2}\right)   & =56d^{2}112d+120,\\
h\left(  \frac{d+1}{2}\right)   & =16d^{2}32d+48,\\
h\left(  \frac{d+3}{2}\right)   & =24\left(  d+1\right)  ^{2}.
\end{align*}
For $p\in\left\{  2,..,d-1\right\}  ,$ the least value of $h$ is $h\left(
\frac{d+1}{2}\right)  $ which is strictly greater than zero, which contradicts
inequality \ref{s6}. Therefore,%
\[
\nabla_{p+1}>4\lambda_{p}\nabla_{p},\ \ \ \ p\in\left\{  1,..,d-1\right\}  .
\]
By putting $p=d$ in \ref{s} and $p=d-1$ in \ref{s7}, we obtain%
\[
\nabla_{d+1}=y\left[  \nabla_{d}-4\lambda_{d-1}\nabla_{d-1}+\frac{16}%
{3}\lambda_{d-1}\lambda_{d-2}\nabla_{d-2}\right]
\]
and%
\[
\nabla_{d}>4\lambda_{d-1}\nabla_{d-1},
\]
which leads to $\nabla_{d+1}>\frac{16}{3}y\lambda_{d-1}\lambda_{d-2}%
\nabla_{d-2},$ and over p and $\nabla_{d+1}$ is strictly positive, which
achieve the proof.
\end{proof}

\end{document}